\documentclass[reqno,a4paper,11pt]{article}
\usepackage{graphics}
\usepackage{graphicx}
\usepackage{floatrow}
\usepackage{float}
\usepackage{pifont}
\usepackage{subcaption}
\usepackage{enumerate}
\usepackage{epsfig}
\usepackage[mathscr]{euscript}
\usepackage{mathrsfs}
\usepackage{amssymb}
\usepackage{mathtools}
\usepackage{amsmath,amsthm,amsfonts}
\usepackage{cases}
\usepackage {indentfirst} 
\usepackage[usenames,dvipsnames]{color}
\usepackage[colorlinks=true, allcolors=blue]{hyperref}

\newtheorem{theorem}{Theorem}[section]

\newtheorem{proposition}[theorem]{Proposition}

\newtheorem{definition}[theorem]{Definition}

\newtheorem{lemma}[theorem]{Lemma}

\newtheorem{remark}[theorem]{Remark}

\begin{document}
\title{Sesqui-regular graphs with fixed smallest eigenvalue}
\author{Jack H. Koolen$^{1,2}$, Brhane Gebremichel$^1$, Jae Young Yang$^3$,\\
Qianqian Yang$^{1}$\footnote{Q. Yang is the corresponding author.} 
\\ \\
\small $^1$ School of Mathematical Sciences,\\
\small University of Science and Technology of China, \\
\small 96 Jinzhai Road, Hefei, 230026, Anhui, PR China\\
\small $^2$ Wen-Tsun Wu Key Laboratory of CAS,\\
\small 96 Jinzhai Road, Hefei, 230026, Anhui, PR China\\
\small ${}^3$ Samsung SDS,\\
\small Olympic-ro 35-gil 125, Songpa-gu, Seoul, 05510, Republic of Korea\\
\small {\tt e-mail: koolen@ustc.edu.cn, brhaneg220@mail.ustc.edu.cn,}\\\vspace{-3pt}
\small {\tt  piez@naver.com, qqyang91@ustc.edu.cn}\vspace{-3pt}
}
\date{}
\maketitle

\begin{abstract}
Let $\lambda\geq2$ be an integer. For strongly regular graphs with parameters $(v, k, a,c)$ and smallest eigenvalue $-\lambda$, Neumaier gave two bounds on $c$ by using algebraic property of strongly regular graphs. In this paper, we will study a new class of regular graphs called sesqui-regular graphs, which contains strongly regular graphs as a subclass, and prove that for a sesqui-regular graph with parameters $(v,k,c)$ and smallest eigenvalue at least $-\lambda$, if $k$ is very large, then either $c \leq \lambda^2(\lambda -1)$ or $v-k-1 \leq \frac{(\lambda-1)^2}{4} + 1$ holds.

\end{abstract}

\textbf{Keywords}: sesqui-regular graph, smallest eigenvalue, Hoffman graphs, Alon-Boppana Theorem 

\textbf{AMS classification}: 05C50, 05C75, 05C62

\section{Introduction}

A {\em strongly regular graph}  with parameters $(v, k, a, c)$ is a $k$-regular graph with $v$ vertices such that the number of common neighbors of any two adjacent vertices is exactly $a$ and the number of common neighbors of any two distinct non-adjacent vertices is exactly $c$.
Neumaier \cite{-m} proved the following two theorems on strongly regular graphs.

\begin{theorem}[{\cite[Theorem 3.1]{-m}}]\label{-m1} 
Let $\lambda\geq 2$ be an integer. For any connected strongly regular graph $G$ with parameters $(v, k, a,c)$, if the smallest eigenvalue of $G$ is $-\lambda$, then either $G$ is a complete multipartite graph, or $c \leq \lambda^3 (2\lambda -3)$.
\end{theorem}

\begin{theorem}[{\cite[Theorem 5.1]{-m}}]\label{-m2} Let $\lambda\geq2$ be an integer. Except for finitely many exceptions, any strongly regular graph with smallest eigenvalue $-\lambda$ is a Steiner graph, a Latin square graph, or a complete multipartite graph.
\end{theorem}

To prove Theorem \ref{-m1} and Theorem \ref{-m2}, Neumaier used the Krein parameters of the underlying association scheme. In \cite{YK18}, Yang and Koolen used a new method and showed the following result on co-edge regular graphs. A {\em co-edge regular graph} with parameters $(v, k, c)$ is a $k$-regular graph with $v$ vertices such that the number of common neighbors of any two distinct non-adjacent vertices  is exactly $c$.  

\begin{theorem}[{\cite[Theorem 7.1]{YK18}}]\label{YK1}
Let $\lambda \geq 2$ be a real number. There exists a real number $M_1(\lambda)$ such that, for any connected co-edge regular graph $G$ with parameters $(v, k, c)$, if the smallest eigenvalue of $G$ is at least $-\lambda $, then either $c\leq M_1(\lambda)$ or $v-k-1 \leq \frac{(\lambda -1)^2}{4}+1$ holds.
\end{theorem}

This result can be regarded as a generalization of Theorem \ref{-m1}.

Before we state our main results, we need to define a larger class of regular graphs, which contains both the class of strongly regular graphs and the class of co-edge regular graphs, that is, the class of sesqui-regular graphs. A {\em sesqui-regular graph} with parameters $(v, k, c)$ is a $k$-regular graph with $v$ vertices such that the number of common neighbors of any two vertices at distance 2 is exactly $c$. 

The first result is a slight generalization of Theorem \ref{YK1} on sesqui-regular graphs. 

\begin{theorem}\label{YK2}
	Let $\lambda \geq 2$ be a real number. There exists a real number $M_2(\lambda)$ such that, for any connected sesqui-regular graph $G$ with parameters $(v, k, c)$, if the smallest eigenvalue of $G$ is at least $-\lambda $, then either $c\leq M_2(\lambda)$ or $v-k-1 \leq \frac{(\lambda -1)^2}{4}+1$ holds.
\end{theorem}

We omit the proof of Theorem \ref{YK2}, since it can be proven by exactly the same method as in the proof of Theorem  \ref{YK1}, for which we refer to \cite{YK18}. Our main result in the present paper is the following.

\begin{theorem}\label{main sesqui}
Let $\lambda \geq 2$ be an integer. There exists a real number $C(\lambda)$ such that, for any connected sesqui-regular graph $G$ with parameters $(v, k, c)$, if the smallest eigenvalue of $G$ is at least  $-\lambda$ and $k \geq C(\lambda)$, then either $c\leq \lambda^2(\lambda-1)$ or $v-k-1\leq \frac{(\lambda -1)^2}{4}+1$ holds.
\end{theorem}

\begin{remark}
	\begin{enumerate}[\rm(i)]
		\item If we replace the condition that $\lambda\geq2$ is an integer in Theorem \ref{main sesqui} by that $\lambda\geq2$ is a real number with $\lambda-\lfloor\lambda\rfloor-\frac{1}{\lambda-1}<0$, then we can replace $c\leq\lambda^2(\lambda-1)$ in the conclusion by $c\leq\lfloor\lambda\rfloor\lfloor\lambda(\lambda-1)\rfloor$.
		\item We wonder whether we can replace $c\leq\lambda^2(\lambda-1)$ in Theorem \ref{main sesqui}  by $c\leq\lambda^2$. Then the result will be sharp, as a Steiner graph with smallest eigenvalue $-\lambda$ has $c=\lambda^2$.
		\item Also, Theorem \ref{main sesqui} can be regarded as a combinatorial generalization of Theorem \ref{-m2}, since Theorem \ref{-m2} states that the only strongly regular graphs with smallest eigenvalue $-\lambda$ and large valency are the Steiner graps with $c=\lambda^2$, the Latin square graphs with $c= \lambda(\lambda -1)$ and the complete multipartite graphs $K_{t\times \lambda}$.
		\end{enumerate}
	\end{remark}

This paper is organized as follows. In Section \ref{sec:Hoffman}, we introduce the basic definitions and properties of Hoffman graphs, quasi-cliques, and associated Hoffman graphs which are main tools of this paper. In Section \ref{sec:few Hoffman graphs}, we give some Hoffman graphs with smallest eigenvalue less than $-\lambda$, which play a key role in our proof. In Section \ref{sec:proof}, we show a proof of Theorem \ref{main sesqui}.

\section{Preliminaries}\label{sec:Hoffman}

\subsection{Definitions and properties related to Hoffman graphs}
In this section, we introduce the definitions and basic properties of Hoffman graphs and associated Hoffman graphs. For more details or proofs, see \cite{Jang, KKY, KCY, Woo}.

\begin{definition}

A Hoffman graph $\mathfrak{h}$ is a pair $(H, \ell)$ with a labeling map $\ell : V(H) \rightarrow \{{ \textbf{\textit{f,s}}}\}$ satisfying two conditions:

\begin{enumerate}[\rm(i)]

\item the vertices with label \textbf{\textit{f}} are pairwise non-adjacent,

\item every vertex with label \textbf{\textit{f}} has at least one neighbor with label \textbf{\textit{s}}.

\end{enumerate}

\end{definition}

The vertices with label  \textbf{\textit{f}} are called {\it fat} vertices, and the set of fat vertices of $\mathfrak{h}$ is denoted by $V_{\textbf{\textit{f}}}(\mathfrak{h})$. The vertices with label \textbf{\textit{s}} are called {\it slim} vertices, and the set of slim vertices is denoted by $V_{\textbf{\textit{s}}}(\mathfrak{h})$. 

For a vertex $x$ of $\mathfrak{h}$, we define $N_{\mathfrak{h}}^{\textbf{\textit{s}}}(x)$ (resp. $N_{\mathfrak{h}}^{\textbf{\textit{f}}}(x)$) the set of slim (resp. fat) neighbors of $x$ in $\mathfrak{h}$. If every slim vertex of $\mathfrak{h}$ has a fat neighbor, then we call $\mathfrak{h}$ \emph{fat}. In a similar fashion, we define $N^{\textbf{\textit{s}}}_\mathfrak{h}(x_1,x_2)$ (resp. $N^{\textbf{\textit{f}}}_\mathfrak{h}(x_1,x_2)$) to be the set of common slim (resp. fat) neighbors of $x_1$ and $x_2$ in $\mathfrak{h}$.

The \emph{slim graph} of the Hoffman graph $\mathfrak{h}$ is the subgraph of $H$ induced on $V_{\textbf{\textit{s}}}(\mathfrak{h})$.

For a fat vertex $F$ of $\mathfrak{h}$, a \emph{quasi-clique} (with respect to $F$) is a subgraph of the slim graph of $\mathfrak{h}$ induced on the slim vertices adjacent to $F$ in $\mathfrak{h}$, and we denote it by $Q_{\mathfrak{h}}(F)$. Now, we give more basic definitions as follows.

\begin{definition} \label{def: induced sub}

A Hoffman graph $\mathfrak{h}_1 = (H_1, \ell_1)$ is called an {\it induced Hoffman subgraph} of a Hoffman graph $\mathfrak{h}=(H,\ell)$, if $H_1$ is an induced subgraph of $H$ and $\ell(x) = \ell_1 (x)$ for all vertices $x$ of $H_1$.

\end{definition}

\begin{definition}

Two Hoffman graphs $\mathfrak{h}=(H, \ell)$ and $\mathfrak{h}'=(H', \ell')$ are called {\it isomorphic} if there exists a graph isomorphism from $H$ to $H'$ which preserves the labeling.

\end{definition}

\begin{definition}For a Hoffman graph $\mathfrak{h}=(H,\ell)$, there exists a matrix $C$ such that the adjacency matrix $A$ of $H$ satisfies
	\begin{eqnarray*}
		A=\left(
		\begin{array}{cc}
			A_{\textbf{\textit{s}}}  & C\\
			C^{T}  & O
		\end{array}
		\right),
	\end{eqnarray*}
	where $A_{\textbf{\textit{s}}}$ is the adjacency matrix of the slim graph of $\mathfrak{h}$, and $O$ is a zero matrix. The special matrix $Sp(\mathfrak{h})$ of $\mathfrak{h}$ is the real symmetric matrix $A_{\textbf{\textit{s}}}-CC^{T}.$
\end{definition}

The \emph{eigenvalues} of $\mathfrak{h}$ are the eigenvalues of its special matrix $Sp(\mathfrak{h})$, and the smallest eigenvalue of $\mathfrak{h}$ is always denoted by  $\lambda_{\min}(\mathfrak{h})$. 

Now, we discuss some spectral properties of the smallest eigenvalue of a Hoffman graph and its induced Hoffman subgraph.

\begin{lemma}[{\cite[Corollary 3.3]{Woo}}] If $\mathfrak{h}'$ is an induced Hoffman subgraph of $\mathfrak{h}$, then $\lambda_{\min}(\mathfrak{h}') \geq \lambda_{\min}(\mathfrak{h})$ holds.

\end{lemma}

Let $\mathfrak{h}$ be a Hoffman graph with $V_{\textbf{\textit{f}}}(\mathfrak{h}) = \{F_1, F_2, \ldots,F_t\}$. For a positive integer $p$, let $G(\mathfrak{h}, p)$ be the graph obtained from $\mathfrak{h}$ by replacing every fat vertex of $\mathfrak{h}$ by a complete graph $K_p$ of $p$ slim vertices, and connecting all vertices of the $K_p$ to all neighbors of the corresponding fat vertex by edges.  For the smallest eigenvalue of $\mathfrak{h}$ and $G(\mathfrak{h}, p)$, we have the following:

\begin{theorem}[Hoffman and Ostrowski]\label{OH}
Let $\mathfrak{h}$ be a Hoffman graph and $p$ a positive integer. Then

$$ \lambda_{\min}(G(\mathfrak{h}, p)) \geq \lambda_{\min}(\mathfrak{h}), $$
and
$$ \lim_{p\rightarrow \infty} \lambda_{\min}(G(\mathfrak{h}, p)) = \lambda_{\min}(\mathfrak{h}). $$

\end{theorem}
For a proof, we refer to \cite[Theorem 2.14]{Jang} and \cite[Theorem 3.2]{KYY3}.

\subsection{Quasi-cliques and associated Hoffman graphs}

For a positive integer $m$,  let $\widetilde{K}_{2m}$ be the graph  on $2m +1$ vertices consisting of a complete graph $K_{2m}$ and a vertex which is adjacent to exactly half of the vertices of $K_{2m}$. 

%

\begin{lemma}[{cf.~\cite[Lemma 3.2]{YK18}}]\label{min2} Let $\lambda\geq 1$ be a real number. There exist minimum positive integers $t(\lambda):=\lfloor\lambda^2\rfloor+1$ and $m(\lambda)$ such that, for any integers $t\geq t(\lambda)$ and $m\geq m(\lambda)$, the smallest eigenvalue of the graph $K_{1,t}$ and the smallest eigenvalue of the graph $\widetilde{K}_{2m}$ both are less than $-\lambda$.\end{lemma}

Let $G$ be a graph that does not contain $\widetilde{K}_{2m}$ as an induced subgraph. For a positive integer $n \geq (m+1)^2$, let $\mathcal{C}(n)$ be the set of maximal cliques of $G$ with at least $n$ vertices. Define the relation $\equiv_n^m$ on $\mathcal{C}(n)$ by $C_1 \equiv_n^m C_2$, if each vertex $x \in C_1$ has at most $m-1$ non-neighbors in $C_2$ and each vertex $y \in C_2$ has at most $m-1$ non-neighbors in $C_1$ for $C_1, C_2 \in \mathcal{C}(n)$. Note that $\equiv_n^m$ on $\mathcal{C}(n)$ is an equivalence relation if $n\geq(m+1)^2$ (see \cite [Lemma 3.1]{KKY}).

For a maximal clique $C \in \mathcal{C}(n)$, let $[C]_n^m$ denote the equivalence class containing $C$ under the equivalence relation $\equiv_n^m$. We can define the term {\it quasi-clique} as follows:

\begin{definition}\label{def: quasi-clique}

Let $m\geq2$ and $n\geq2$ be two integers where $n \geq (m+1)^2$, and let $G$ be a graph that does not contain $\widetilde{K}_{2m}$ as an induced subgraph. For a maximal clique $C \in \mathcal{C}(n)$, the quasi-clique $Q([C]_n^m)$, with respect to the pair $(m,n)$, is the induced subgraph of $G$ on the vertices which have at most $m-1$ non-neighbors in $C$.

\end{definition}

To check the well-definedness of the quasi-clique $Q([C]_n^m)$ for $C \in \mathcal{C}(n)$, we refer the readers to {\rm \cite[Lemma 3.2 and Lemma 3.3]{KKY}}. The reason to define a quasi-clique in Definition \ref{def: quasi-clique} is to construct a Hoffman graph in Definition \ref{def: asso hoff} which is highly related to a given graph. After Definition \ref{def: asso hoff}, we will explain that the quasi-clique with respect to a pair $(m,n)$ in Definition \ref{def: quasi-clique}  essentially coincides with the quasi-clique with respect to a fat vertex defined before Definition \ref{def: induced sub}.

\begin{definition}\label{def: asso hoff}

Let $m\geq2$ and $n\geq2$ be two integers where $n \geq (m+1)^2$ and let $G$ be a graph which does not contain $\widetilde{K}_{2m}$ as an induced subgraph. Let $[C_1]_n^m, [C_2]_n^m, \ldots, [C_t]_n^m$ be all the equivalence classes of $G$ under $\equiv_n^m$. The associated Hoffman graph $\mathfrak{g} = \mathfrak{g}(G, m, n)$ is the Hoffman graph with the following properties:

\begin{enumerate}[\rm(i)]

\item $V_{\textbf{\textit{s}}}(\mathfrak{g}) = V(G)$ and $V_{\textbf{\textit{f}}}(\mathfrak{g}) = \{F_1, \dots, F_t\}$, where $t$ is the number of the equivalence classes of $G$ under $\equiv_n^m$,

\item the slim graph of $\mathfrak{g}$ is isomorphic to $G$,

\item the fat vertex $F_i$ is adjacent to every vertex of the quasi-clique $Q([C_i]_n^m)$ for $i=1,2,\dots,t$.
\end{enumerate}

\end{definition}

From the above definition of associated Hoffman graphs, we find that for each $i=1,\dots,t$, the quasi-clique $Q([C_i]_n^m)$ with respect to the pair $(m, n)$ is exactly the quasi-clique $Q_{\mathfrak{g}}(F_i)$ in $\mathfrak{g} = \mathfrak{g}(G, m, n)$ with respect to the fat vertex $F_i$.

The following proposition shows an important property of the associated Hoffman graph.

\begin{proposition}[{\cite[Proposition 4.1]{KKY}}]\label{asso}
	Let $m \geq 2, \phi,\sigma,p \geq 1$ be integers.
	There exists a positive integer $n = n(m, \phi, \sigma,  p) \geq (m+1)^2$ such that, for any graph $G$, any integer $n' \geq n$ and any Hoffman graph $\mathfrak{h}$ with at most $\phi$ fat vertices and at most $\sigma$ slim vertices, the graph  $G(\mathfrak{h}, p)$ is an induced subgraph of $G$, provided that the graph $G$ satisfies the following conditions:
	\begin{enumerate}[\rm(i)]
		\item the graph $G$ does not contain  $\widetilde{K}_{2m}$ as an induced subgraph,
		\item its associated Hoffman graph $\mathfrak{g} = \mathfrak{g}(G, m, n')$ contains $\mathfrak{h}$ as an induced Hoffman subgraph.
	\end{enumerate}
\end{proposition}

\section{Some Hoffman graphs}\label{sec:few Hoffman graphs}

Let $\lambda\geq2$ be a real number. In this section, we study some Hoffman graphs with smallest eigenvalue less than $-\lambda$, which will be used in the proof of Theorem \ref{main sesqui} later.

For a graph $H$, let $\mathfrak{q}(H)$ be the Hoffman graph obtained by attaching one fat vertex to all vertices of $H$. Then it is easily checked that $\lambda_{\min}(\mathfrak{q}(H)) = -\lambda_{\max}(\overline{H})-1$ holds, where $\lambda_{\max}(\overline{H})$ is the largest eigenvalue of the complement $\overline{H}$ of $H$. 

\begin{lemma}[{\cite[Lemma 3.1]{YK18}}]\label{l+2}

Let $\lambda \geq 2$ be a real number. Let $H$ be a graph with $\lfloor (\lambda-1)^2\rfloor +2$ vertices which has at least one isolated vertex. Then $\lambda_{\min}(\mathfrak{q}(H))< -\lambda$. 
\end{lemma}

\begin{lemma}\label{pp}

Let $\lambda  \geq 2$ be a real number. Let $\{H(\lambda)_i\}_{i=1}^{r(\lambda)}$ be the set of all graphs with $\lfloor (\lambda-1)^2\rfloor+2$ vertices which have at least one isolated vertex. There exists a positive integer $p'(\lambda)$ such that, for every integer $p'\geq p'(\lambda)$, the inequality $\lambda_{\min}(G(\mathfrak{q}(H(\lambda)_i), p'))< -\lambda$ holds for all $i= 1, \dots, r(\lambda)$.

\end{lemma}

\begin{proof}

For each $i$, we have $\lambda_{\min}(\mathfrak{q}(H(\lambda)_i)) < -\lambda$ by Lemma \ref{l+2}. Then for 
$i=1, \dots, r(\lambda)$, there exist positive integers $p'_i(\lambda)$'s such that for every integer $p_i'\geq p'_i(\lambda)$, \[\lambda_{\min}(G(\mathfrak{q}(H(\lambda)_i), p'_i) )< -\lambda\] 
holds by Theorem \ref{OH}. By taking $p'(\lambda) = \max\limits_i p'_i(\lambda)$, we complete the proof.\end{proof}

\begin{remark} 
 Let $H$ be the graph $K_{\lfloor (\lambda-1)^2\rfloor +1} \cup K_1$ and let $H'$ be any graph with ${\lfloor (\lambda-1)^2\rfloor+2}$ vertices which has at least one isolated vertex. It can be shown that $\lambda_{\min}(G(\mathfrak{q}(H), p')) \geq \lambda_{\min}(G(\mathfrak{q}(H'), p'))$ holds for every integer $p'\geq p'(\lambda)$.
\end{remark}

Now, we introduce three Hoffman graphs $\mathfrak{h}^{(t)}$, $\mathfrak{h}^{(t,1)}$ and $\mathfrak{c}_t$. The Hoffman graph $\mathfrak{h}^{(t)}$ is the Hoffman graph with one slim vertex adjacent to $t$ fat vertices. The Hoffman graph $\mathfrak{h}^{(t,1)}$ is the Hoffman graph with two adjacent slim vertices $s_1$ and $s_t$ such that $s_t$ is adjacent to $t$ fat vertices and $s_1$ is adjacent to one fat vertex different from the $t$ fat neighbors of $s_t$ (see Figure \ref{fig:cherry}). Note that $\lambda_{\min}(\mathfrak{h}^{(t)}) = -t$ and $\lambda_{\min}(\mathfrak{h}^{(t,1)}) = \frac{-t-1-\sqrt{t^2-2t+5}}{2}$.

\begin{figure}[h]
\centering
\includegraphics[scale=1.2]{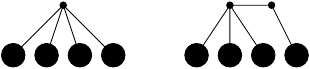}
\caption{The Hoffman graphs $\mathfrak{h}^{(4)}$ and $\mathfrak{h}^{(3,1)}$}
\label{fig:cherry}
\end{figure}
\noindent The Hoffman graph $\mathfrak{c}_t$ is the Hoffman graph obtained by attaching one fat vertex to $t$ vertices of a $K_{t+1}$ (see Figure \ref{fig:kbb}). Then $\lambda_{\min}(\mathfrak{c}_t) = \frac{-1-\sqrt{1+4t}}{2}$.

\begin{figure}[H]
\centering
\includegraphics[scale=1.2]{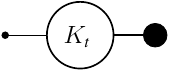}
\caption{The Hoffman graph $\mathfrak{c}_t$}
\label{fig:kbb}
\end{figure}

For the smallest eigenvalue of $\mathfrak{h}^{(t)}$, $\mathfrak{h}^{(t,1)}$ and $\mathfrak{c}_t$, we have the following.

\begin{lemma}\label{ppp}
	
	Let $\lambda \geq 2$ be a real number with $\lambda-\lfloor\lambda\rfloor-\frac{1}{\lambda-1}<0$. There exists a positive integer $p''(\lambda)$ such that for every integer $p''\geq p''(\lambda) $, the following three inequalities hold.
	
	\begin{enumerate}[\rm(i)]
		
		\item $\lambda_{\min}(G(\mathfrak{h}^{(\lfloor\lambda\rfloor+1)}, p''))< -\lambda$, 
		
		\item $\lambda_{\min}(G(\mathfrak{h}^{(\lfloor\lambda\rfloor,1)}, p''))< -\lambda$,
		
		\item $\lambda_{\min}(G(\mathfrak{c}_{\lfloor\lambda^2 -\lambda\rfloor+1}, p''))<-\lambda$.
		
	\end{enumerate}
	
\end{lemma}

\begin{proof}
	It is straightforward to obtain \[\lambda_{\min}(\mathfrak{h}^{(\lfloor\lambda\rfloor+1)}) < -\lambda,~  \lambda_{\min}(\mathfrak{h}^{(\lfloor\lambda\rfloor,1)}) < -\lambda,\text{ and }\lambda_{\min}(\mathfrak{c}_{\lfloor\lambda^2 -\lambda\rfloor+1}) <-\lambda.\]  Theorem \ref{OH} implies that there exist positive integers $p''_1(\lambda)$, $p''_2(\lambda)$ and $p''_3(\lambda)$ such that, for any integers
	$p''_1\geq p''_1(\lambda)$, $p''_2\geq p''_2(\lambda)$ and $p''_3\geq p''_3(\lambda)$,
	\[\lambda_{\min}(G(\mathfrak{h}^{(\lfloor\lambda\rfloor+1)},p''_1))< -\lambda,~\lambda_{\min}(G(\mathfrak{h}^{(\lfloor\lambda\rfloor,1)}, p''_2))< -\lambda,\]
	and  \[\lambda_{\min}(G(\mathfrak{c}_{\lfloor\lambda^2 -\lambda\rfloor+1}, p''_3))<-\lambda\]
	hold.
	By setting $p''(\lambda) = \max\limits_i p''_i(\lambda)$, we complete the proof. 
\end{proof}

\section{Proof of the main theorem}\label{sec:proof}







In this section, we will complete the proof of Theorem \ref{main sesqui}. Before that, the following theorem is necessary.

\begin{theorem}\label{thm:pre} Let $\lambda\geq2$ be an integer and $m(\lambda)$ the minimum positive integer such that for any integer $m\geq m(\lambda)$, the graph $\widetilde{K}_{2m}$ has smallest eigenvalue less than $-\lambda$. There exists a positive integer $n'\geq (m(\lambda)+1)^2$ such that for any integer $q\geq n'$, a real number $C'(\lambda,q)$ satisfying the following exists.
	
For any connected sesqui-regular graph $G$ with parameters $(v,k,c)$, where $v-k-1>\frac{(\lambda-1)^2}{4}+1$, and with smallest eigenvalue at least $-\lambda$, if $k\geq C'(\lambda,q)$, then the associated Hoffman graph $\mathfrak{g}:=\mathfrak{g}(G,m(\lambda),q)$ is fat and has the following properties:
\begin{enumerate}[\rm(i)]
	\item $\mathfrak{g}$ has $G$ as its slim graph.
	\item $|N_\mathfrak{g}^{\textbf{\textit{f}}}(x)|\leq \lambda$ for every $x\in V_{\textbf{\textit{s}}}(\mathfrak{g})$, and if the equality holds, then for each slim vertex $y$ adjacent to $x$, $|N_\mathfrak{g}^{\textbf{\textit{f}}}(x,y)|\geq1$.
	\item For $x\in V_{\textbf{\textit{s}}}(\mathfrak{g})$ and $F\in V_{\textbf{\textit{f}}}(\mathfrak{g})$, if $x$ is adjacent to $F$, then $x$ has at most $(\lambda-1)^2$ non-neighbors in the quasi-clique $Q_\mathfrak{g}(F)$.
	\item For every $F\in V_{\textbf{\textit{f}}}(\mathfrak{g})$, the quasi-clique $Q_\mathfrak{g}(F)$ is a clique.
	\item For $x\in V_{\textbf{\textit{s}}}(\mathfrak{g})$ and $F\in V_{\textbf{\textit{f}}}(\mathfrak{g})$, if $x$ is not adjacent to $F$, then $x$ has at most $\lambda^2-\lambda$ neighbors in the quasi-clique $Q_\mathfrak{g}(F)$.
	\item For $x\in V_{\textbf{\textit{s}}}(\mathfrak{g})$, there exist a slim vertex $y\in V_{\textbf{\textit{s}}}(\mathfrak{g})$ at distance $2$ from $x$ in $G$ and  a fat vertex $F'\in N_\mathfrak{g}(y)$ such that for any slim vertex $z$ in $N_\mathfrak{g}^{\textbf{\textit{s}}}(x,y)$, either $|N_\mathfrak{g}^{\textbf{\textit{f}}}(x,z)|\geq1$ or $z$ is adjacent to $F'$.
\end{enumerate}
\end{theorem}
\begin{proof}
By Theorem \ref{YK2}, we obtain $c\leq M_2(\lambda)$ immediately, where $M_2(\lambda)$ is such that Theorem \ref{YK2} holds. 

Let $p'(\lambda)$ and $p''(\lambda)$ be the integers in Lemma \ref{pp} and Lemma \ref{ppp} respectively, and let $p =\max\{p'(\lambda), p''(\lambda)\}$. Let  $t:=t(\lambda)=\lambda^2+1$, where $t(\lambda)$ is the integer in Lemma \ref{min2}. Let $n=n(m(\lambda), \lambda+1, \lambda^2 -\lambda+2, p)$, where $n(m(\lambda), \lambda+1, \lambda^2 -\lambda+2, p)$ is the integer in Proposition \ref{asso}. Let $n' = \max\{n, M_2(\lambda) + 2(\lambda-1)^2 + 3\}$. For an integer $q\geq n'$, let $C'(\lambda,q)=R((\lambda^2-\lambda) R(q-1,t), t)$, where $R(,)$ denotes the Ramsey number of two positive integers. 

Suppose $k \geq C'(\lambda,q)$.  Note that $G$ contains neither $\widetilde{K}_{2m(\lambda)}$ nor $K_{1,t}$ as an induced subgraph, since 
	$G$ has smallest eigenvalue at least $-\lambda$, but both  $\widetilde{K}_{2m(\lambda)}$ and $K_{1,t}$ have smallest eigenvalue less than $-\lambda$ by Lemma \ref{min2}. By the property of Ramsey numbers, every vertex of $G$ lies in a clique of size $(\lambda^2-\lambda)R(q-1,t)+1~(\geq q)$. This means that in the associated Hoffman graph $\mathfrak{g}:=\mathfrak{g}(G,m(\lambda),q)$, every slim vertex has a fat neighbor, and lies in at least one quasi-clique with size at least $ (\lambda^2-\lambda) R(q-1,t)+1$.  By Proposition \ref{asso}, Lemma \ref{pp} and Lemma \ref{ppp}, the associated Hoffman graph $\mathfrak{g}$ does not contain Hoffman graphs in the set $\{\mathfrak{h}^{(\lambda+1)}, \mathfrak{h}^{(\lambda, 1)}, \mathfrak{c}_{\lambda^2 -\lambda+1}\}\cup \{\mathfrak{q}(H(\lambda)_i)\}_{i=1}^{r(\lambda)}$ as induced Hoffman subgraphs. We will show that $\mathfrak{g}$ satisfies {\rm (i)}--{\rm (v)}.

{\rm (i)} This follows from the definition of associated Hoffman graphs directly.

{\rm (ii)} If $|N_\mathfrak{g}^{\textbf{\textit{f}}}(x')|>\lambda$ for some $x'\in V_{\textbf{\textit{s}}}(\mathfrak{g})$, then $\mathfrak{g}$ will contain $\mathfrak{h}^{(\lambda+1)}$ as an induced Hoffman subgraph. This leads to a contradiction. Suppose that a vertex $x$ with exactly $\lambda$ fat neighbors has a slim neighbor $y$ which satisfies $|N_\mathfrak{g}^{\textbf{\textit{f}}}(x,y)|=0$. Then $\mathfrak{g}$ will contain $\mathfrak{h}^{(\lambda, 1)}$ as an induced Hoffman subgraph, as $\mathfrak{g}$ is fat. This leads to a contradiction again. 

{\rm (iii)} If $x$ has $(\lambda-1)^2+ 1$ non-neighbors in $Q_\mathfrak{g}(F)$, then $\mathfrak{g}$ will contain one of $\mathfrak{q}(H(\lambda)_i)$'s as an induced Hoffman subgraph. This gives a contradiction.

{\rm (iv)}  By the definition of $\mathfrak{g}$, we have that the size $|N_\mathfrak{g}^{\textbf{\textit{s}}}(F)|$ of the quasi-clique $Q_\mathfrak{g}(F)$ satisfies $|N_\mathfrak{g}^{\textbf{\textit{s}}}(F)|\geq q\geq n'\geq M_2(\lambda) + 2(\lambda-1)^2+3$.  Suppose that there are two non-adjacent vertices $x$ and $y$ in $Q_\mathfrak{g}(F)$. Then both $x$ and $y$ have at most $(\lambda-1)^2$ non-neighbors in $Q_\mathfrak{g}(F)$ by {\rm (iv)}. This implies that they have more than $M_2(\lambda)$ common neighbors in $Q_\mathfrak{g}(F)$, which contradicts the assumption $c \leq M_2(\lambda)$.  

{\rm (v)} By {\rm (iv)}, the quasi-clique $Q_\mathfrak{g}(F)$ is a clique. If $x$ has $\lambda^2 -\lambda+ 1$ neighbors in $Q_\mathfrak{g}(F)$, then $\mathfrak{g}$ will contain $\mathfrak{c}_{\lambda^2 -\lambda+1}$ as an induced Hoffman subgraph. This gives a contradiction.

{\rm (vi)} Assume $N_\mathfrak{g}^{\textbf{\textit{f}}}(x)=\{F_1,F_2,\ldots, F_s\}$. Let $T=N_\mathfrak{g}^{\textbf{\textit{s}}}(x)-\{w\mid w$ is adjacent to one of $ F_i$ for $i=1,\ldots, s\}$. If $T=\emptyset$, then {\rm(vi)} follows immediately. So we may assume $|T|\geq1$. If $|T|\geq R(q-1, t)$, then $T$ contains a clique of size at least $q-1$. By the definition of associated Hoffman graphs, there will be one more new fat vertex with respect to a quasi-clique containing $x$ and a clique of size at least $q-1$ in $T$ which can be attached to the vertex $x$. This contradicts the fact that $x$ has exactly $s$ fat neighbors. So we obtain $1\leq |T| \leq R(q-1,t)-1$. Let $w'$ be a vertex in $T$. We have concluded that $w'$ lies in a quasi-clique with size at least $(\lambda^2-\lambda)R(q-1,t)+1$ before. Let $F'$ be the fat vertex with respect to this quasi-clique. Note that for each vertex $w\in(T-N_\mathfrak{g}^{\textbf{\textit{s}}}(F'))\cup\{x\}$, $w$ has at most $\lambda^2 -\lambda$ neighbors in the quasi-clique $Q_\mathfrak{g}(F')$ by {\rm (v)}. Thus there exists at least one vertex in $Q_\mathfrak{g}(F')$ which has no neighbors in $(T-N_\mathfrak{g}^{\textbf{\textit{s}}}(F'))\cup\{x\}$, as  $|N_\mathfrak{g}^{\textbf{\textit{s}}}(F')|>(\lambda^2 - \lambda)(R(q-1, t)-1)+(\lambda^2 - \lambda)\geq(\lambda^2 - \lambda)|(T-N_\mathfrak{g}^{\textbf{\textit{s}}}(F'))\cup\{x\}|$. Let $y$ be such a vertex in $Q_\mathfrak{g}(F')$. This shows the existence.
\end{proof}

Now, we are in the position to prove Theorem \ref{main sesqui}.

\vspace{0.2cm}
\noindent{\it Proof of Theorem \ref{main sesqui}.} Let $C(\lambda):=C'(\lambda,n')$, where $n'$ and  $C'(\lambda,n')$ are such that Theorem \ref{thm:pre} holds. Let $G$ be a connected sesqui-regular graph with parameters $(v,k,c)$, where $k\geq C(\lambda)$, and with smallest eigenvalue at least $-\lambda$.  We may assume $v-k-1>\frac{(\lambda-1)^2}{4}+1$. By Theorem \ref{thm:pre}, there exists a fat Hoffman graph $\mathfrak{g}$ which satisfies properties {\rm (i)--(vi)} in Theorem \ref{thm:pre}.

Let $x$ be a slim vertex of $\mathfrak{g}$. Suppose $N_\mathfrak{g}^{\textbf{\textit{f}}}(x)=\{F_1,F_2,\ldots,F_s\}$, where $s \leq \lambda$ by Theorem \ref{thm:pre} {\rm (ii)}. Let $y$ be a vertex at distance $2$ from $x$ in $G$ and $F'$ a fat neighbor of $y$ such that Theorem \ref{thm:pre} {\rm (vi)} holds. By Theorem \ref{thm:pre} {\rm (iv)}, $y$ does not lie in any of the quasi-cliques $Q_\mathfrak{g}(F_i)$'s.  Now we look at the set $N_\mathfrak{g}^{\textbf{\textit{s}}}(x,y)$.

First, consider the case where every slim neighbor of $x$ lies in one of the quasi-clique $Q_\mathfrak{g}(F_i)$. By Theorem \ref{thm:pre} {\rm (v)}, $y$ has at most $\lambda^2 - \lambda$ neighbors in each $Q_\mathfrak{g}(F_i)$. Therefore, $c=|N_\mathfrak{g}^{\textbf{\textit{s}}}(x,y)|\leq s(\lambda^2-\lambda)\leq\lambda(\lambda^2-\lambda)=\lambda^2(\lambda-1)$.

Now suppose that there are slim neighbors of $x$ which do not lie in any of $Q_\mathfrak{g}(F_i)$'s. Then $s \leq \lambda-1$ by Theorem \ref{thm:pre} {\rm (ii)}. Theorem \ref{thm:pre} {\rm (vi)} says that every vertex $z\in N_\mathfrak{g}^{\textbf{\textit{s}}}(x,y)$ lies in either one of $Q_\mathfrak{g}(F_i)$'s or $Q_\mathfrak{g}(F')$.  Note that $x$ has at most $\lambda^2 - \lambda$ neighbors in $Q_\mathfrak{g}(F')$, and $y$ has at most $\lambda^2 - \lambda$ neighbors in each $Q_\mathfrak{g}(F_i)$, thus $c=|N_\mathfrak{g}^{\textbf{\textit{s}}}(x,y)|\leq (\lambda^2-\lambda)+s(\lambda^2-\lambda)\leq\lambda(\lambda^2-\lambda)=\lambda^2(\lambda-1)$.

This completes the proof.
\qed

\section*{Acknowledgements}

J.H. Koolen is partially supported by the National Natural Science Foundation of China (No. 12071454), Anhui Initiative in Quantum Information Technologies (No. AHY150000) and the project ``Analysis and Geometry on Bundles'' of Ministry of Science and Technology of the People’s
Republic of China.

B. Gebremichel is supported by the Chinese Scholarship Council at USTC, China.

Q. Yang is partially supported by the Fellowship of China Postdoctoral Science Foundation (No. 2020M671855).

We are also grateful to the referee for his/her careful reading and valuable comments. 


\end{document}